\documentclass[a4paper,11pt]{article}
\usepackage[T1]{fontenc}
\usepackage{amsfonts}
\usepackage{verbatim}
\usepackage{amsmath, amsthm, amssymb}

\title{Three Upsilon Transforms Related to Tempered Stable Distributions}
\author{Michael Grabchak\footnote{Email address: mgrabcha@uncc.edu}, {\it University of North Carolina Charlotte}}

\begin{document}
\newtheorem{prop}{Proposition}
\newtheorem{thrm}[prop]{Theorem}
\newtheorem{defn}[prop]{Definition}
\newtheorem{cor}[prop]{Corollary}
\newtheorem{lemma}[prop]{Lemma}

\newcommand{\rd}{\mathrm d}
\newcommand{\rE}{\mathrm E}
\newcommand{\ts}{TS^p_\alpha}
\newcommand{\tr}{\mathrm{tr}}
\newcommand{\iid}{\stackrel{\mathrm{iid}}{\sim}}
\newcommand{\eqd}{\stackrel{d}{=}}
\newcommand{\cond}{\stackrel{d}{\rightarrow}}
\newcommand{\conv}{\stackrel{v}{\rightarrow}}
\newcommand{\conw}{\stackrel{w}{\rightarrow}}
\newcommand{\conp}{\stackrel{p}{\rightarrow}}
\newcommand{\confdd}{\stackrel{fdd}{\rightarrow}}
\newcommand{\plim}{\mathop{\mathrm{p\mbox{-}lim}}}
\newcommand{\dlim}{\operatorname*{d-lim}}

\maketitle

\begin{abstract}

We discuss the properties of three upsilon transforms, which are related to the class of $p$-tempered $\alpha$-stable ($TS^p_\alpha$) distributions. In particular, we characterize their domains and show how they can be represented as compositions of each other.  Further, we show that if $-\infty<\beta<\alpha<2$ and $0<q<p<\infty$ then they can be used to transform the L\'evy measures of $TS^p_\beta$ distributions into those of $TS^q_\alpha$.\\


\end{abstract}

\section{Introduction}

Over the past decade there has been considerable interest in the study of transforms of L\'evy measures, especially upsilon transforms, which are closely related to stochastic integration with respect to L\'evy processes. Although upsilon transforms were formally defined in \cite{Barndorff-Nielsen:Rosinski:Thorbjornsen:2008}, the concept goes back, at least, to \cite{Jurek:1990}. In this paper, we study the properties of three upsilon transforms, which are related to tempered stable distributions.

Let $\mathfrak M_{\sigma f}$ be the collection of all $\sigma$-finite Borel measures on $\mathbb R^d$ such that every $M\in\mathfrak M_{\sigma f}$ satisfies $M(\{0\})=0$, and let $\rho$ be a nonzero $\sigma$-finite Borel measure on $(0,\infty)$. A mapping $\Upsilon_\rho:\mathfrak M_{\sigma f}\mapsto\mathfrak M_{\sigma f}$ is called an upsilon transform with dilation measure $\rho$ 
if, for any $M\in\mathfrak M_{\sigma f}$, we have
\begin{eqnarray}
[\Upsilon_\rho M](B) = \int_0^\infty M(s^{-1}B)\rho(\rd s),
\quad B\in\mathfrak B(\mathbb R^d), 
\end{eqnarray}
where $\mathfrak B(\mathbb R^d)$ refers to the Borel sets in $\mathbb R^d$. Theorem 3.2 in \cite{Barndorff-Nielsen:Rosinski:Thorbjornsen:2008} tells us that if $\Upsilon_\rho M$ is a L\'evy measure then, necessarily, $M$ is a L\'evy measure as well. However, $\Upsilon_\rho M$ need not be a L\'evy measure even if $M$ is. We write $\mathfrak D(\Upsilon_\rho)$ to denote the collection of all L\'evy measures for which $\Upsilon_\rho M$ remains a L\'evy measure. This is called the domain of $\Upsilon_\rho$. Further, we write $\mathfrak R(\Upsilon_\rho)$ to denote the collection of all L\'evy measures $M$ for which there exists an $M'\in\mathfrak D(\Upsilon_\rho)$ with $M=\Upsilon_\rho M'$. This is called the range of $\Upsilon_\rho$.

A probabilistic interpretation of $\Upsilon_\rho$ is given in \cite{Barndorff-Nielsen:Rosinski:Thorbjornsen:2008}. Specifically, for an upsilon transform  $\Upsilon_\rho$ let $\eta_\rho(t) = \rho([t,\infty))$ for $t>0$. Assume that $\eta_\rho(t)<\infty$ for each $t>0$ and  let $\eta_\rho^*$ be the inverse of $\eta_\rho$ in the sense $\eta_\rho^*(t) = \inf\{s>0: \eta_\rho(s)\le t\}$ for $t>0$. Let $\{X_t:t\ge0\}$ be a L\'evy process such that the distribution of $X_1$ has L\'evy measure $M$. Define (if possible) the stochastic integral
$$
Y=\int_0^{\eta_{\rho}(0)} \eta^*_\rho(t) \rd X_t
$$
in the sense of \cite{Sato:2007}. If the integral exists then $M\in\mathfrak D(\Upsilon_\rho)$ and the distribution of $Y$ is infinitely divisible with L\'evy measure $\Upsilon_\rho M$.  However, even if $M\in\mathfrak D(\Upsilon_\rho)$ this does not guarantee that the integral exists since we must to be careful with the Gaussian part and the shift.  

In this paper we focus on upsilon transforms with the following dilation measures:\\
1. For $\alpha\in\mathbb R$ and $p>0$ let
\begin{eqnarray}
\psi_{\alpha,p}(\rd s)= s^{-\alpha-1}e^{-s^p}1_{s>0}\rd s.
\end{eqnarray}
2. For  $-\infty<\beta<\alpha<\infty$ and $p>0$ let
\begin{eqnarray}
\tau_{\beta\to\alpha,p} (\rd s)= \frac{1}{K_{\alpha,\beta,p}}s^{-\alpha-1}(1-s^p)^{(\alpha-\beta)/p-1}1_{0<s<1}\rd s,
\end{eqnarray}
where
$$
K_{\alpha,\beta,p}=\int_0^\infty u^{\alpha-\beta-1} e^{-u^p}\rd u = p^{-1}\Gamma\left(\frac{\alpha-\beta}{p}\right). 
$$
3. For $0<q<p<\infty$ and $\alpha\in\mathbb R^d$ let
\begin{eqnarray}
\pi_{\alpha,p\to q}(\rd s) = pf_{q/p}(s^{-p})s^{-\alpha-p-1}1_{s>0}\rd s,
\end{eqnarray}
where, for $r\in(0,1)$, $f_r$ is the density of a fully right skewed $r$-stable distribution with Laplace transform
\begin{eqnarray}\label{eq: laplace stable}
\int_0^\infty e^{-tx}f_r(x)\rd x = e^{-t^r}.
\end{eqnarray}

For simplicity of notation we write
$$
\Psi_{\alpha,p}=\Upsilon_{\psi_{\alpha,p}},\ \mathfrak T_{\beta\to\alpha,p} = \Upsilon_{\tau_{\beta\to\alpha,p} }, \mbox{ and } \mathfrak P_{\alpha,p\to q} = \Upsilon_{\pi_{\alpha,p\to q} }.
$$
The transform $\Psi_{\alpha,p}$ was first introduced in \cite{Maejima:Nakahara:2009} and then further studied in \cite{Maejima:Ueda:2010} and \cite{Maejima:Perez-Abreu:Sato:2013}. Several important subclasses were studied in \cite{Sato:2006}, \cite{Barndorff-Nielsen:Maejima:Sato:2006}, and \cite{Jurek:2007}. The transform, $\mathfrak T_{\alpha\to\beta,p}$ was discussed in Section 4 of \cite{Maejima:Perez-Abreu:Sato:2013}, and the case where $p=1$ was considered in \cite{Sato:2006}, \cite{Sato:2010}, and \cite{Jurek:2014}. The transform $\mathfrak P_{\alpha,p\to q}$ is essentially new, although it appears, implicitly, in \cite{Grabchak:2012}. Further, a related transform is studied in \cite{Barndorff-Nielsen:Thorbjornsen:2006}.

The transform  $\Psi_{\alpha,p}$ is closely related to the class of tempered stable distributions, which is a class of models that is obtained by modifying the tails of infinite variance stable distributions to make them lighter.  These models were first introduced in \cite{Rosinski:2007}. The more general class of $p$-tempered $\alpha$-stable distributions ($TS^p_\alpha$), where $p>0$ and $\alpha<2$, was introduced in \cite{Grabchak:2012} as a class of infinitely divisible distributions with no Gaussian part and a L\'evy measure of the form $\Psi_{\alpha,p} M$, where $M\in\mathfrak D(\Psi_{\alpha,p})$.  If we allow these distributions to have a Gaussian part then we get the class of models studied in \cite{Maejima:Nakahara:2009}. This, in turn, contains important subclasses including the Thorin class, the Goldie-Steutel-Bondesson class, the class of type M distributions, and the class of generalized type G distributions. For more about tempered stable distributions and their use in a variety of application areas see  \cite{Rachev:Kim:Bianchi:Fabozzi:2011}, \cite{Grabchak:2015a}, \cite{Grabchak:2015b}, and the references therein. 

The relationship between tempered stable distributions and the transforms $\mathfrak T_{\beta\to\alpha,p}$ and $\mathfrak P_{\alpha,p\to q} $ will become apparent from studying the relationships among the transforms. Several such relationships  are known. Specifically, if $-\infty<\gamma<\beta<\alpha<\infty$ then Theorem 3.1 in \cite{Sato:2006} (see also \cite{Jurek:2014}) implies that 
$$
\Psi_{\alpha,1}
= \mathfrak T_{\beta\to\alpha,1} \Psi_{\beta,1}
$$ 
and Theorem 4.7 in \cite{Sato:2010} implies that 
$$
\mathfrak T_{\gamma\to\alpha,1} = \mathfrak T_{\beta\to\alpha,1} \mathfrak T_{\gamma\to\beta,1}.
$$
We will show that these relations hold with $1$ replaced by any $p>0$. Further, we show that if $0<r<q<p<\infty$ and $\alpha\in\mathbb R$ then
$$
\Psi_{\alpha,q} = \mathfrak P_{\alpha,p\to q}\Psi_{\alpha,p}
$$
and
$$
\mathfrak P_{\alpha,p\to r} =\mathfrak P_{\alpha,q\to r}\mathfrak P_{\alpha,p\to q}.
$$
Putting these together implies that if $-\infty<\beta<\alpha<\infty$ and $0<q<p<\infty$ then
$$
\Psi_{\alpha,q} = \mathfrak P_{\alpha,p\to q}\mathfrak T_{\beta\to\alpha,p}\Psi_{\beta,p} =  \mathfrak T_{\beta\to\alpha,q}\mathfrak P_{\beta,p\to q}\Psi_{\beta,p}.
$$
Thus we can transform $\Psi_{\beta,p}$ into $\Psi_{\alpha,q}$ by using the other two transforms. In the context of tempered stable distributions this means that we can transform the L\'evy measures of $TS^p_\beta$ distributions into those of $TS^q_\alpha$.

\section{Main Results}

We begin by characterizing the domains of the transforms of interest. Toward this end we introduce some notation. For $\alpha\in[0,2]$ let $\mathfrak M^\alpha$ be the class of Borel measures on $\mathbb R^d$ such that $M\in\mathfrak M^\alpha$ if and only if
\begin{eqnarray}
M(\{0\})=0 \mbox{ and } \int_{\mathbb R^d}\left(|x|^2\wedge |x|^\alpha\right)M(\rd x)<\infty.
\end{eqnarray}
Note that if $0<\alpha_1<\alpha_2<2$ then $\mathfrak M^{2}\subsetneq\mathfrak M^{\alpha_2}\subsetneq\mathfrak M^{\alpha_1}\subsetneq \mathfrak M^0$, and that the class $\mathfrak M^0$ is the class of all L\'evy measures on $\mathbb R^d$. Let $\mathfrak M^{\log}$ be the subclass of $\mathfrak M^0$ such that $M\in\mathfrak M^{\log}$ satisfies $\int_{|x|>1}\log|x|M(\rd x)<\infty$.

\begin{thrm}\label{thrm: domains}
1. For $-\infty<\beta<\alpha<\infty$ and $p>0$ we have\\
$$
\mathfrak D(\mathfrak T_{\beta\to\alpha,p})=\mathfrak D(\Psi_{\alpha,p})=\left\{
\begin{array}{ll}
\mathfrak M^0 & \mbox{if }\alpha<0\\
\mathfrak M^{\log}& \mbox{if }\alpha=0\\
\mathfrak M^\alpha& \mbox{if }\alpha\in(0,2)\\
\{0\} & \mbox{if }\alpha\ge2
\end{array}
\right..
$$
2. For $\alpha\in\mathbb R$ and $0<q<p<\infty$ we have\\
$$
\mathfrak D(\mathfrak P_{\alpha,p\to q})=\left\{
\begin{array}{ll}
\mathfrak M^0 & \mbox{if }\alpha-q<0\\
\mathfrak M^{\log}& \mbox{if }\alpha-q=0\\
\mathfrak M^{\alpha-q}& \mbox{if }\alpha-q\in(0,2)\\
\{0\} & \mbox{if }\alpha-q\ge2
\end{array}
\right..
$$
\end{thrm}

The proof follows from a general result and is given in Section \ref{sec: proofs}. We note that $\mathfrak D(\Psi_{\alpha,p})$ was already fully characterized in  \cite{Maejima:Nakahara:2009}.  Henceforth, we assume that $\alpha<2$ in the case of $\mathfrak T_{\beta\to\alpha,p}$ and $\Psi_{\alpha,p}$ and that $\alpha<2+q$ in the case of $\mathfrak P_{\alpha,p\to q}$. Of course, our results will, trivially, remain true for the other cases. 

We now turn to the composition of transforms. Let $\Upsilon_{\rho_1},\Upsilon_{\rho_2}$ be two upsilon transforms and define the composition $\Upsilon_{\rho_2}\Upsilon_{\rho_1}$ on the domain 
$$
\mathfrak D(\Upsilon_{\rho_2}\Upsilon_{\rho_1})=\{M\in\mathfrak D(\Upsilon_{\rho_1}):\Upsilon_{\rho_1}M\in\mathfrak D(\Upsilon_{\rho_2})\}.
$$
Proposition 4.1 in \cite{Barndorff-Nielsen:Rosinski:Thorbjornsen:2008} tells us that $\mathfrak D(\Upsilon_{\rho_2}\Upsilon_{\rho_1})=\mathfrak D(\Upsilon_{\rho_1}\Upsilon_{\rho_2})$ and that 
$$
\Upsilon_{\rho_2}\Upsilon_{\rho_1}=\Upsilon_{\rho_1}\Upsilon_{\rho_2}.
$$ 
Thus compositions of upsilon transforms commute. To give a better understanding of the domains of compositions we give the following.

\begin{lemma}\label{lemma: ranges}
If $-\infty<\gamma<\beta<\alpha<\infty$ and $0<r<q<p<\infty$ then 
\begin{eqnarray*}
\mathfrak R(\mathfrak T_{\beta\to\alpha,p}) &\subset& \mathfrak D(\Psi_{\beta,p}),\\
\mathfrak R(\mathfrak T_{\beta\to\alpha,p}) &\subset& \mathfrak D(\mathfrak T_{\gamma\to\beta,p}),\\
\mathfrak R(\Psi_{\alpha,p}) &\subset& \mathfrak D(\mathfrak P_{\alpha,p\to q}),\\
\mathfrak R(\mathfrak P_{\alpha,q\to r}) &\subset& \mathfrak D(\mathfrak P_{\alpha,p\to q}),\\
\mathfrak R(\mathfrak T_{\beta\to\alpha,p}) &\subset& \mathfrak D(\mathfrak P_{\alpha,p\to q}).
\end{eqnarray*}
\end{lemma}

The proof follows from a general result and is given in Section \ref{sec: proofs}. We now state our main result.

\begin{thrm}\label{thrm: composition}
1. If $-\infty<\beta<\alpha<2$ and $p>0$ then
\begin{eqnarray*}
\Psi_{\alpha,p} = \mathfrak T_{\beta\to\alpha,p}\Psi_{\beta,p}.
\end{eqnarray*}
2. If $-\infty<\gamma<\beta<\alpha<2$ and $p>0$ then
\begin{eqnarray*}
\mathfrak T_{\gamma\to\alpha,p} = \mathfrak T_{\beta\to\alpha, p} \mathfrak T_{\gamma\to\beta,p}.
\end{eqnarray*}
3. If $\alpha<2$ and $0<q<p<\infty$ then
\begin{eqnarray*}
\Psi_{\alpha,q} = \mathfrak P_{\alpha,p\to q}\Psi_{\alpha,p}.
\end{eqnarray*}
4. If $0<r<q<p<\infty$ and $-\infty<\alpha<2+r$ then
$$
\mathfrak P_{\alpha,p\to r} = \mathfrak P_{\alpha,q\to r}\mathfrak P_{\alpha,p\to q}.
$$
\end{thrm}

The proof is given in Section \ref{sec: proofs}. In all cases equality of domains is part of the result. Further, in the above, all compositions commute. Before proceeding, we recall a result from \cite{Maejima:Nakahara:2009}.  

\begin{prop}\label{prop: onto on psi}
If $\alpha<2$ and $p>0$ then the transform $\Psi_{\alpha,p}$ is one-to-one.
\end{prop}

Combining this with Theorem \ref{thrm: composition} will give the following.

\begin{cor}\label{cor: one-to-one}
If $\beta<\alpha<2$ and $p>0$ then the transform $\mathfrak T_{\beta\to\alpha,p}$ is one-to-one. If $0<q<p$ and $\alpha<0$ then the transform $\mathfrak P_{\alpha,p\to q}$ is one-to-one. 
\end{cor}

For $\mathfrak T_{\beta\to\alpha,p}$ a different proof was given in \cite{Maejima:Perez-Abreu:Sato:2013}. For $\mathfrak P_{\alpha,p\to q}$, the case where $\alpha\ge0$ is more complicated and will be dealt with in a future work.

\begin{proof}
We begin with Part 1. Let $M,M'\in\mathfrak D(\mathfrak T_{\beta\to\alpha,p})=\mathfrak D(\Psi_{\alpha,p})$. If $\mathfrak T_{\beta\to\alpha,p}M=\mathfrak T_{\beta\to\alpha,p}M'$ then $\Psi_{\beta,p}\mathfrak T_{\beta\to\alpha,p}M=\Psi_{\beta,p}\mathfrak T_{\beta\to\alpha,p}M'$ and hence by commutativity and Theorem \ref{thrm: composition} we have $\Psi_{\alpha,p}M=\Psi_{\alpha,p}M'$. From here Proposition \ref{prop: onto on psi} implies that $M=M'$ and hence $\mathfrak T_{\beta\to\alpha,p}$ is one-to-one. The proof of Part 2 is similar. We just need to note that, in this case, $\mathfrak D(\Psi_{\alpha,q})=\mathfrak M^0$. 
\end{proof}

We now interpret Theorem \ref{thrm: composition} in the context of tempered stable distributions. For $\alpha<2$ and $p>0$ let $LTS^p_\alpha$ be the class of L\'evy measures of $p$-tempered $\alpha$-stable distributions, and note that $LTS^p_\alpha=\mathfrak R(\Psi_{\alpha,p})$. For $-\infty<\beta<\alpha<2$ and $0<q<p<\infty$ let $\mathfrak T_{\beta\to\alpha,p}^{TS}$ and $\mathfrak P^{TS}_{\alpha,p\to q}$ be the restrictions of $\mathfrak T_{\beta\to\alpha,p}$ and $\mathfrak P_{\alpha,p\to q}$ to the domains $LTS^p_\beta\cap \mathfrak D(\mathfrak T_{\beta\to\alpha,p})$ and $LTS^p_\alpha$ respectively. Note that, by Lemma \ref{lemma: ranges}, $LTS^p_\alpha\subset \mathfrak D(\mathfrak P_{\alpha,p\to q})$.

\begin{cor}
For $\beta<\alpha<2$ and $p>0$ the mapping  $\mathfrak T_{\beta\to\alpha,p}^{TS}$ is a bijection from $LTS^p_\beta\cap \mathfrak D(\mathfrak T_{\beta\to\alpha,p})$ onto $LTS^p_\alpha$. For $0<q<p$ and $\alpha<2$ the mapping $\mathfrak P^{TS}_{\alpha,p\to q}$ is a bijection from $LTS^p_\alpha$ onto $LTS^q_\alpha$.
\end{cor}

\begin{proof}
The result is immediate from Theorem \ref{thrm: composition} and Corollary \ref{cor: one-to-one}, except in the case of $\mathfrak P^{TS}_{\alpha,p\to q}$ with $\alpha\in[0,2)$. In this case we can show that $\mathfrak P^{TS}_{\alpha,p\to q}$  is one-to-one by arguments similar to the proof of Corollary \ref{cor: one-to-one}.  
\end{proof}

\section{Proofs}\label{sec: proofs}

In this section we prove our main results. First, recall that, for $r\in(0,1)$, $f_r$ is the probability density of a fully right-skewed $r$-stable distribution with Laplace transform given by \eqref{eq: laplace stable}. 

\begin{lemma}\label{lemma: stable density}
1. If $r\in(0,1)$ then there is a $K>0$ depending on $r$ such that
\begin{eqnarray*}
f_r(x) \sim Kx^{-r-1}\mbox{ as } x\to\infty.
\end{eqnarray*}
2. If $r\in(0,1)$ and $\beta\in(-\infty,r)$ then
\begin{eqnarray*}
\int_0^\infty s^{\beta} f_r(s)\rd s<\infty.
\end{eqnarray*}
3. If $r,p\in(0,1)$ then
$$
f_{rp} (u ) = \int_0^\infty f_r(uy^{-1/r}) y^{-1/r}f_p(y)\rd y.
$$
\end{lemma}

\begin{proof}
Part 1 follows from (14.36) in \cite{Sato:1999}. When $\beta<0$ Part 2 follows from Theorem 5.4.1 in \cite{Uchaikin:Zolotarev:1999} and when $\beta\in[0,r)$ it follows from Part 1. Now, let $X\sim f_r$ and $Y\sim f_p$ be independent random variables. The fact that
$$
\rE\left[e^{-tY^{1/r}X}\right] = \rE\left[\rE\left[e^{-tY^{1/r}X}|Y\right]\right] =  \rE\left[e^{-t^rY}\right] = e^{-t^{rp}}
$$
implies that $Y^{1/r}X\sim f_{rp}$. From here Part 3 follows by representing the density of  $Y^{1/r}X$ in terms of the densities of $X$ and $Y$.
\end{proof}

\begin{lemma}\label{lemma: domains}
Assume that $\rho(\rd s) = g(s)1_{s>0}\rd s$ and that there exist $\delta\in(0,1)$, $\alpha\in\mathbb R$, and $0<a<b<\infty$ such that $a<s^{\alpha+1}g(s)<b$ for all $s\in(0,\delta)$. When $\alpha<2$ assume also that $\int_0^\infty s^2 g(s)\rd s<\infty$. In this case
$$
\mathfrak D(\Upsilon_\rho)=\left\{
\begin{array}{ll}
\mathfrak M^0 & \mbox{if }\alpha<0\\
\mathfrak M^{\log} & \mbox{if }\alpha=0\\
\mathfrak M^\alpha & \mbox{if }\alpha\in(0,2)\\
\{0\} & \mbox{if }\alpha\ge2
\end{array}
\right..
$$
Further, when $\alpha\in(0,2)$ we have $\mathfrak R(\Upsilon_{\rho})\subset\mathfrak M^\beta$ for every $\beta\in[0,\alpha)$.
\end{lemma}

We note that a related result is given in Theorem 4.1 of \cite{Sato:2010}.

\begin{proof}
Fix $M\in\mathfrak M^0$. We need to characterize when 
$$
\int_{\mathbb R^d}\left(|x|^2\wedge1\right)[\Upsilon_\rho M](\rd x)<\infty.
$$
First assume $\alpha\ge2$. If $M\ne0$ then there exists a $\delta'\in(0,\delta)$ such that $M(|x|\le1/\delta')>0$ and 
\begin{eqnarray*}
\int_{|x|\le 1}|x|^2[\Upsilon_\rho M](\rd x)&=&\int_{\mathbb R^d}|x|^2\int_0^{1/|x|} s^2 g(s)\rd sM(\rd x)\\
&\ge& \int_{|x|\le1/\delta'}|x|^2\int_0^{\delta'} s^2 g(s)\rd sM(\rd x)\\
&\ge& a\int_{|x|\le1/\delta'}|x|^2\int_0^{\delta'} s^{1-\alpha} \rd sM(\rd x)=\infty.
\end{eqnarray*}
Now assume that $\alpha<2$. We have
\begin{eqnarray*}
&&\int_{|x|\le1}|x|^2[\Upsilon_\rho M](\rd x)=\int_{\mathbb R^d}|x|^2\int_0^{1/|x|} s^2 g(s)\rd sM(\rd x)\\
&&\ \le \int_{|x|\le1/\delta}|x|^2M(\rd x)\int_0^{\infty} s^2 g(s)\rd s + b\int_{|x|>1/\delta}|x|^2\int_0^{1/|x|} s^{1-\alpha}\rd sM(\rd x)\\
&&\  = \int_{|x|\le1/\delta}|x|^2M(\rd x)\int_0^{\infty} s^2 g(s)\rd s + \frac{1}{2-\alpha}\int_{|x|>1/\delta}|x|^\alpha M(\rd x)<\infty
\end{eqnarray*}
and
\begin{eqnarray*}
\int_{|x|>1}[\Upsilon_\rho M](\rd x)&=& \int_{|x|\le1/\delta}\int_{1/|x|}^\infty g(s)\rd sM(\rd x)\\
&&\qquad + \int_{|x|>1/\delta}\int_{1/|x|}^\delta g(s)\rd sM(\rd x)\\
&&\qquad + \int_{|x|>1/\delta}\int_{\delta}^\infty g(s)\rd sM(\rd x) =: I_1+I_2+I_3.
\end{eqnarray*}
We have $I_{3}<\infty$ since $\int_{\delta}^\infty g(s)\rd s\le \delta^{-2}\int_0^\infty s^2g(s)\rd s<\infty$ and 
\begin{eqnarray*}
I_1\le \int_{|x|\le1/\delta}|x|^2M(\rd x)\int_0^\infty s^2 g(s)\rd s<\infty.
\end{eqnarray*}
Now note that
$$
a\int_{|x|>1/\delta}\int_{1/|x|}^\delta s^{-1-\alpha}\rd s M(\rd x) \le I_{2} \le  b\int_{|x|>1/\delta}\int_{1/|x|}^\delta s^{-1-\alpha}\rd s M(\rd x).
$$
From here the fact that $\int_{1/|x|}^\delta s^{-1-\alpha}\rd s = (|x|^\alpha-\delta^{-\alpha})/\alpha$ when $\alpha\ne0$ and it equals 
$\log|x\delta|$ when $\alpha=0$ completes the proof of the first part.

Now assume that $\alpha\in(0,2)$ and $\beta\in[0,\alpha)$. It suffices to show that for any $M\in\mathfrak M^\alpha$
$$
\int_{|x|>1} |x|^\beta [\Upsilon_\rho M](\rd x)= \int_{\mathbb R^d}|x|^\beta \int_{|x|^{-1}}^\infty s^{\beta}g(s)\rd sM(\rd x) <\infty.
$$
Observing that
$$
\int_{|x|\le1/\delta}|x|^\beta \int_{|x|^{-1}}^\infty s^{\beta}g(s)\rd sM(\rd x) \le \int_{|x|\le1/\delta}|x|^2M(\rd x) \int_{0}^\infty s^{2}g(s)\rd s<\infty,
$$
\begin{eqnarray*}
\int_{|x|>1/\delta}|x|^\beta \int_{\delta}^\infty s^{\beta}g(s)\rd sM(\rd x) <\infty,
\end{eqnarray*}
and
\begin{eqnarray*}
\int_{|x|>1/\delta}|x|^\beta \int_{|x|^{-1}}^\delta s^{\beta}g(s)\rd sM(\rd x) &\le& b\int_{|x|>1/\delta}|x|^\beta \int_{|x|^{-1}}^\infty s^{\beta-\alpha-1}\rd sM(\rd x)\\
&=& \frac{b}{\alpha-\beta}\int_{|x|>1/\delta}|x|^\alpha M(\rd x)<\infty
\end{eqnarray*}
gives the result.
\end{proof}

We can now prove Theorem \ref{thrm: domains} and Lemma \ref{lemma: ranges}.

\begin{proof}[Proof of Theorem \ref{thrm: domains} and Lemma \ref{lemma: ranges}]
Both results follow easily from Lemma \ref{lemma: domains}, we just need to check that the assumptions hold. We only verify this for $\mathfrak P_{\alpha,p\to q}$ as it is immediate for the other cases.  In this case we have $g(s) = pf_{q/p}(s^{-p})s^{-\alpha-p-1}$.  Lemma \ref{lemma: stable density} implies that
$$
\int_0^\infty s^2 g(s)\rd s=p\int_0^\infty f_{q/p}(s^{-p})s^{1-\alpha-p}\rd s= \int_0^\infty f_{q/p}(v)v^{-(2-\alpha)/p}\rd v<\infty
$$
and that $g(s)\sim pK s^{q-\alpha-1}$ as $s\downarrow0$. From here the result follows.
\end{proof}

\begin{proof}[Proof of Theorem \ref{thrm: composition}]
In all cases, equality of the domains follows from Theorem \ref{thrm: domains} and Lemma \ref{lemma: ranges}. We now turn to proving the equalities. We begin with Part 1. Fix $M\in\mathfrak D(\mathfrak T_{\beta\to\alpha,p}\Psi_{\beta,p})$ and let $M' = \mathfrak T_{\beta\to\alpha,p}\Psi_{\beta,p}M$. For any $B\in\mathfrak B(\mathbb R^d)$
\begin{eqnarray*}
M'(B) &=&   K_{\alpha,\beta,p}^{-1}\int_0^1 [\Psi_{\beta,p}M](u^{-1}B) u^{-\alpha-1}\left(1-u^p\right)^{\frac{\alpha-\beta}{p}-1}\rd u\\
&=&K_{\alpha,\beta,p}^{-1}\int_0^\infty\int_0^1M((ut)^{-1}B) t^{-1-\beta}e^{-t^p} u^{-\alpha-1}\left(1-u^p\right)^{\frac{\alpha-\beta-p}{p}}\rd u\rd t\\
&=& K_{\alpha,\beta,p}^{-1}\int_0^\infty\int_0^tM(v^{-1}B) t^{\alpha-\beta-1}e^{-t^p}v^{-\alpha-1}\left(1-\frac{v^p}{t^p}\right)^{\frac{\alpha-\beta-p}{p}}\rd v\rd t\\
&=& K_{\alpha,\beta,p}^{-1}\int_0^\infty\int_v^\infty M(v^{-1}B) t^{p-1}e^{-t^p}
v^{-\alpha-1}\left(t^p-v^p\right)^{\frac{\alpha-\beta-p}{p}}\rd t\rd v\\
&=& K_{\alpha,\beta,p}^{-1}\int_0^\infty M(v^{-1}B) e^{-v^p}v^{-\alpha-1}\rd v \int_0^\infty
e^{-s^p} s^{\alpha-\beta-1}\rd s\\
&=& \int_0^\infty M(v^{-1}B) e^{-v^p}v^{-\alpha-1}\rd v = [\Psi_{\alpha,p}M](B),
\end{eqnarray*}
where the third line follows by the substitution $v=ut$ and the fifth by the substitution $s^p = t^p-v^p$

We now show Part 2. Note that by the well-known relationship between beta and gamma functions
\begin{eqnarray*}
\int_0^1w^{\alpha-\beta-1}\left(1-w^p\right)^{\frac{\beta-\gamma}{p}-1}\rd w &=& p^{-1} \int_0^1w^{\frac{\alpha-\beta}{p}-1}\left(1-w\right)^{\frac{\beta-\gamma}{p}-1}\rd w\\
&=& p^{-1}\frac{\Gamma\left(\frac{\alpha-\beta}{p}\right)\Gamma\left(\frac{\beta-\gamma}{p}\right)}{\Gamma\left(\frac{\alpha-\gamma}{p}\right)} = \frac{K_{\alpha,\beta,p}K_{\beta,\gamma,p}}{K_{\alpha,\gamma,p}}.
\end{eqnarray*}
For simplicity of notation let $A=K^{-1}_{\alpha,\beta,p}K^{-1}_{\beta,\gamma,p}$ and note that
$$
K_{\alpha,\gamma,p}^{-1}=A\int_0^1w^{\alpha-\beta-1}\left(1-w^p\right)^{\frac{\beta-\gamma}{p}-1}\rd w.
$$
Fix $M\in\mathfrak D(\mathfrak T_{\beta\to\alpha, p} \mathfrak T_{\gamma\to\beta,p})$ and let $M' = \mathfrak T_{\beta\to\alpha, p} \mathfrak T_{\gamma\to\beta,p}M$. For $B\in\mathfrak B(\mathbb R^d)$
\begin{eqnarray*}
&&M'(B) =   K^{-1}_{\alpha,\beta,p}\int_0^1[\mathfrak T_{\gamma\to\beta,p}M](u^{-1}B) u^{-\alpha-1}\left(1-u^p\right)^{\frac{\alpha-\beta}{p}-1}\rd u\\
&&\ =   A\int_0^1 \int_0^1 M((ut)^{-1}B)  u^{-\alpha-1}\left(1-u^p\right)^{\frac{\alpha-\beta}{p}-1}\rd u t^{-\beta-1}\left(1-t^p\right)^{\frac{\beta-\gamma}{p}-1}\rd t\\
&&\ =  A\int_0^1 \int_0^t M(v^{-1}B) v^{-\alpha-1}\left(1-\frac{v^p}{t^p}\right)^{\frac{\alpha-\beta}{p}-1}\rd v t^{\alpha-\beta-1}\left(1-t^p\right)^{\frac{\beta-\gamma}{p}-1}\rd t\\
&&\ = A\int_0^1 M(v^{-1}B)  v^{-\alpha-1}\int_v^1\left(t^p-v^p\right)^{\frac{\alpha-\beta}{p}-1}\left(1-t^p\right)^{\frac{\beta-\gamma}{p}-1}t^{p-1}\rd t\rd v\\
&&\ =  A\int_0^1 M(v^{-1}B)  v^{-\alpha-1}(1-v^p)^{\frac{\alpha-\gamma}{p}-1}\rd v\int_0^1w^{\alpha-\beta-1}\left(1-w^p\right)^{\frac{\beta-\gamma}{p}-1}\rd w \\
&&\ =   K^{-1}_{\alpha,\gamma,p}\int_0^1 M(v^{-1}B)  v^{-\alpha-1}(1-v^p)^{(\alpha-\gamma)/p-1}\rd v = [\mathfrak T_{\gamma\to\alpha,p}M](B),
\end{eqnarray*}
where the third line follows by the substitution $v=ut$ and the fifth by the substitution $w^p=(t^p-v^p)/(1-v^p)$, which implies $1-t^p=(1-w^p)(1-v^p)$.

To show Part 3, fix $M\in\mathfrak D(\mathfrak P_{\alpha,p\to q} \Psi_{\alpha,p})$ and note that for $B\in\mathfrak B(\mathbb R^d)$
\begin{eqnarray*}
[\mathfrak P_{\alpha,p\to q} \Psi_{\alpha,p}M](B) &=&p \int_0^\infty f_{q/p}(s^{-p}) s^{-\alpha-p-1} [\Psi_{\alpha,p}M](s^{-1}B)\rd s\\
&=& \int_0^\infty f_{q/p}(s) s^{\alpha/p} [\Psi_{\alpha,p}M](s^{1/p}B)\rd s\\
&=& \int_0^\infty f_{q/p}(s) s^{\alpha/q}\int_0^\infty M(s^{1/q}t^{-1}B)t^{-1-\alpha}e^{-t^p}\rd t\rd s\\
&=&   \int_0^\infty M(v^{-1}B) v^{-1-\alpha}\int_0^\infty e^{-v^qs} f_{q/p}(s)\rd s\rd v \\
&=&\int_0^\infty M(v^{-1}B) v^{-1-\alpha}e^{-v^p}\rd v = [\Psi_{\alpha,p}M](B) ,
\end{eqnarray*}
where the fourth line follows by the substitution $v=s^{-1/p}t$.

For Part 4, fix $M\in\mathfrak D(\mathfrak P_{\alpha,q\to r}\mathfrak P_{\alpha,p\to q})$ and let $M'=\mathfrak P_{\alpha,q\to r}\mathfrak P_{\alpha,p\to q} M$. For any  $B\in\mathfrak B(\mathbb R^d)$
\begin{eqnarray*}
M'(B) &=&  \int_0^\infty f_{r/q}(t) t^{\alpha/q} [\mathfrak P_{\alpha,p\to q}M](t^{1/q}B)\rd t\\
&=&  \int_0^\infty f_{q/p}(s) s^{\alpha/p}\int_0^\infty f_{r/q}(t) t^{\alpha/q} M(t^{1/q}s^{1/p}B)\rd t\rd s\\
&=&  \int_0^\infty M(u^{1/p}B) u^{\alpha/p} \int_0^\infty f_{q/p}(ut^{-p/q})  f_{r/q}(t) t^{-p/q}\rd t\rd u\\
&=&  \int_0^\infty M(u^{1/p}B) u^{\alpha/p} f_{r/p}(u)\rd u=[\mathfrak P_{\alpha,p\to r} M](B),
\end{eqnarray*}
where the third line follows by the substitution $u=st^{p/q}$ and the fourth by Lemma \ref{lemma: stable density}.
\end{proof}


\begin{thebibliography}{10}

\bibitem{Barndorff-Nielsen:Maejima:Sato:2006}
O.~E. Barndorff-Nielsen, M.~Maejima, and K.~Sato (2006).
\newblock Some classes of multivariate infinitely divisible distributions admitting stochastic integral representations.
\newblock {\em Bernoulli}, 12(1):1--33.

\bibitem{Barndorff-Nielsen:Rosinski:Thorbjornsen:2008}
O.~Barndorff-Nielsen, J.~Rosi\'nski, and S.~Thorbj{\o}rnsen (2008).
\newblock General $\Upsilon$-transforms.
\newblock {\em ALEA Latin American Journal of Probability and Mathematical Statistics}, 4:131--165.

\bibitem{Barndorff-Nielsen:Thorbjornsen:2006}
O.~Barndorff-Nielsen and S.~Thorbj{\o}rnsen (2006).
\newblock Regularizing mappings of L\'evy measures.
\newblock {\em Stochastic Processes and their Applications}, 116(3):423--446.

\bibitem{Grabchak:2012}
M.~Grabchak (2012).
\newblock On a new class of tempered stable distributions: Moments and regular
  variation.
\newblock {\em Journal of Applied Probability}, 49(4):1015--1035.

\bibitem{Grabchak:2015a}
M.~Grabchak (2015a).
\newblock Inversions of {L}\'evy measures and the relation between long and short time behavior of {L}\'evy processes.
\newblock {\em Journal of Theoretical Probability}, 28(1):184--197.

\bibitem{Grabchak:2015b}
M.~Grabchak (2015b).
\newblock On the Consistency of the MLE for Ornstein-Uhlenbeck and Other Selfdecomposable Processes. 
\newblock {\em Statistical Inference for Stochastic Processes}, DOI 10.1007/s11203-015-9118-9.

\bibitem{Jurek:1990}
Z.~J.~Jurek (1990).
\newblock On L\'evy (spectral) measures of integral form on Banach spaces.
\newblock {\em Probability and Mathematical Statistics}, 11(1):139--148.

\bibitem{Jurek:2007}
Z.~J.~Jurek (2007).
\newblock Random integral representations for free-infinitely divisible and tempered stable distributions.
\newblock {\em Statistics \& Probability Letters}, 77(4):417--425.

\bibitem{Jurek:2014}
Z.~J.~Jurek (2014).
\newblock Remarks on the factorization property of some random integrals.
\newblock {\em Statistics \& Probability Letters}, 94:192--195.

\bibitem{Maejima:Nakahara:2009}
M.~Maejima and G.~Nakahara (2009).
\newblock A note on new classes of infinitely divisible distributions on {$\mathbb R^d$}.
\newblock {\em Electronic Communications in Probability}, 14:358--371.

\bibitem{Maejima:Ueda:2010}
M.~Maejima and Y.~Ueda (2010).
\newblock Compositions of mappings of infinitely divisible distributions with applications to finding the limits of some nested subclasses.
\newblock {\em Electronic Communications in Probability}, 15:227--239.

\bibitem{Maejima:Perez-Abreu:Sato:2013}
M.~Maejima, V.~P\'erez-Abreu, K.~Sato (2013).
\newblock  L\'evy meaures involving a generalized form of fractional integrals.
\newblock {\em Probability and Mathematical Statistics}, 33(1):45--63.

\bibitem{Rachev:Kim:Bianchi:Fabozzi:2011}
S.~T.~Rachev, Y.~S.~Kim, M.~L.~Bianchi, and F.~J.~Fabozzi (2011).
\newblock  {\em Financial Models with L\'evy Processes and Volatility Clustering}.
\newblock Wiley, Hoboken, NJ.

\bibitem{Rosinski:2007}
J.~Rosi{\'n}ski (2007).
\newblock Tempering stable processes.
\newblock {\em Stochastic Processes and their Applications}, 117(6):677--707.

\bibitem{Sato:1999}
K.~Sato (1999).
\newblock {\em {L}\'evy Processes and Infinitely Divisible Distributions}.
\newblock Cambridge University Press, Cambridge.

\bibitem{Sato:2006}
K.~Sato (2006).
\newblock Two families of improper stochastic integrals with respect to L\'evy processes.
\newblock {\em ALEA Latin American Journal of Probability and Mathematical Statistics}, 1:47--87.

\bibitem{Sato:2007}
K.~Sato (2007).
\newblock Transforms of infinitely divisible distributions via improper stochastic integrals.
\newblock {\em ALEA Latin American Journal of Probability and Mathematical Statistics}, 3:67--110.

\bibitem{Sato:2010}
K.~Sato (2010).
\newblock Fractional integrals and extensions of selfdecomposability.
\newblock {\em Lecture Notes in Math},  2001:1-91.

\bibitem{Uchaikin:Zolotarev:1999}
V.~V.~Uchaikin and V.~M.~Zolotarev (1999).
\newblock {\em Chance and Stability: Stable Distributions and their Applications}.
\newblock VSP BV, Utrecht.

\end{thebibliography}
\end{document}